\documentclass[a4paper]{article}
\usepackage[english]{babel}
\usepackage{amscd}
\usepackage{amsmath}
\usepackage{amsthm}
\usepackage{amssymb}
\usepackage{graphicx, enumitem}

\usepackage[left=2.5cm,right=2.5cm,
    top=3cm,bottom=3.5cm,bindingoffset=0cm]{geometry}

\usepackage{wrapfig}
\usepackage[section]{placeins}
\usepackage{epigraph}
\usepackage{longtable}
\newlist{longenum}{enumerate}{5}
\setlist[longenum,1]{label=\alph*)}
\usepackage{titlesec}

\usepackage{mathtools}

\usepackage[numbers, sort&compress]{natbib}

\theoremstyle{plain}
    \newtheorem{lemma}{Lemma}[section]
    \newtheorem{theorem}{Theorem}
    \newtheorem{consequence}{Corollary}[section]
    \newtheorem{statement}{Proposition}[section]

\theoremstyle{definition}
    
      \newtheorem{remark}{Remark}[section]

\newcommand{\Ker}[1]{\mathrm{Ker} \, #1}

\newcommand{\trdeg}[1]{\mathrm{tr.deg.} \, #1}

\newcommand{\corank}[1]{\mathrm{corank} \, #1}

\newcommand{\ind}[1]{\mathrm{ind} \, #1}

\newcommand{\diff}[1]{\mathrm{d}  #1}
\newcommand{\diffFX}[2]{ \dfrac{\partial #1}{\partial #2} }

\newcommand{\K}{\mathbb{K}}

\newcommand{\Complex}{\mathbb{C}}
\newcommand{\Sing}{\mathrm{Sing}}

\newcommand{\T}{\mathrm{T}}

\newcommand{\g}{\mathfrak{g}}

\newcommand{\so}{\mathfrak{so}}

\newcommand{\ad}{\mathrm{ad}}

\newcommand{\F}{\mathcal{F}}

	\title{Generalized argument shift method and complete commutative subalgebras in polynomial Poisson algebras}
	\author{Anton Izosimov\footnote{Moscow State University and Higher School of Economics. E-mail: a.m.izosimov@gmail.com}}
	\date{}
\begin{document}
\maketitle
\abstract{The Mischenko-Fomenko argument shift method allows to construct commutative subalgebras in the symmetric algebra $S(\g)$ of a finite-dimensional Lie algebra $\g$. For a wide class of Lie algebras, these commutative subalgebras appear to be complete, i.e. they have maximal transcendence degree. However, for many algebras, Mischenko-Fomenko subalgebras are incomplete or even empty. In this case, we suggest a natural way how to extend Mischenko-Fomenko subalgebras, and give a completeness criterion for these extended subalgebras.}
\section{Introduction}
Let $\g$ be a finite-dimensional Lie algebra. For the sake of simplicity we shall assume that the ground field is $\Complex$, however everything works over an arbitrary field of characteristic zero. The symmetric algebra $S(\g)$ can be naturally identified with the algebra of polynomials on the dual space $\g^*$ and carries a natural Poisson bracket (\textit{Lie-Poisson bracket}) which is defined on linear functions by $\{\xi,\eta\} = [\xi,\eta]$ and is extended to all polynomials by the Leibnitz identity. \par
We will be interested in commutative subalgebras in $S(\g)$. Let $C \subset S(\g)$ be a commutative subalgebra. Then the maximal possible transcendence degree of $C$ is
$$
b(\g) =  \frac{1}{2}(\dim \g + \ind \g).
$$
If $\trdeg C = b(\g)$, then $C$ is called a \textit{complete commutative subalgebra}. Each complete commutative subalgebra in $S(\g)$ can be interpreted as an integrable system on the Poisson manifold $\g^*$.\par
One of the most universal methods for constructing ``large'' commutative subalgebras in $S(\g)$ is the so-called \textit{argument shift method}.
This method was introduced by Mischenko and Fomenko \cite{MF} as a generalization of the Manakov construction \cite{Manakov} for the Lie algebra $\so(n)$.\par
The argument shift method can be described as follows. Let $a \in \g^*$ be an arbitrary regular element. Then there exist $m = \ind \g$ analytic functionally-independent invariants of the coadjoint representation defined in a small neighborhood of $a$. Denote these invariants by $f_1, \dots, f_m$. For each $i = 1, \dots, m$ expand the function $f_i(a+\lambda x)$ in powers of $\lambda$:
$$
f_i(a+\lambda x) = \sum_{j=0}^\infty f_{ij}(x)\lambda^j
$$
where all functions $f_{ij}(x)$ are polynomials. Denote the algebra generated by all these polynomials by $\F_a$ . Then, as was proved by Mischenko and Fomenko, $\F_a$ is a commutative subalgebra in $S(\g)$. Moreover, if $\g$ is semisimple, then $\F_a$ is complete.\par
%Note that the argument shift method was inspired by the work of Manakov \cite{Manakov} on the integration of the multidimensional rigid body equations. 
Note that our description of the argument shift method is slightly different from the original  description which required that the invariants $f_1, \dots, f_m$ are polynomial. The modification of the argument shift method presented here is due to Brailov (see Bolsinov \cite{BPC}).\par

The completeness criterion for sublagebras $\F_a$ was found by Bolsinov \cite{Bolsinov}. Namely, he proved that $\F_a$ is complete if and only if the set of singular elements $\Sing \subset \g^*$ has codimension at least $2$ (if the ground field $\K$ is not algebraically closed, one should consider the singular set in $\g^* \otimes \overline \K$ where $\overline \K$ is the algebraic closure of $\K$; see Bolsinov and Zhang \cite{BZ} for details). 
 \par

Although the family $\F_a$ is in general not complete, Mischenko and Fomenko stated the following conjecture: for each finite-dimensional Lie algebra $\g$, there exists a complete commutative subalgebra $C \subset S(\g)$. This conjecture was proved by Sadetov \cite{Sadetov}, see also Bolsinov \cite{bolsinov2012complete}. However, Sadetov's construction is essentially different from the argument shift method. In particular, Sadetov's subalgebras are commutative only with respect to the standard Poisson structure, while Mischenko-Fomenko subalgebras $\F_a$ have the following remarkable property: they are commutative with respect to two Poisson structures, one of which is standard, and the second one (the so-called \textit{frozen argument bracket}) is defined as follows. It is given on linear functions by $\{\xi,\eta\}_a =  \langle a, [\xi,\eta] \rangle$, where $\langle \, ,\rangle$ denotes the pairing between $\g$ and $\g^*$, and is extended to all polynomials by the Leibnitz identity.\par

Our aim is to show that when the Mischenko-Fomenko subalgebra $\F_a$ is incomplete, i.e. when the singular set $\Sing$ has codimension one, then there is a natural extension $\widetilde \F_a \supset \F_a$ which is also commutative with respect to both brackets $\{f,g\}$ and $\{f,g\}_a$, and to give a completeness criterion for $\widetilde \F_a$.\par
In their paper \cite{BZh}, Bolsinov and Zhang formulated the following ``generalized argument shift'' conjecture: {for each Lie algebra $\g$ there exists a subalgebra in $S(\g)$ which is commutative with respect to both brackets $\{f,g\}$ and $\{f,g\}_a$}. In this way, our note is a step towards the proof of this conjecture.

\section{Generalized argument shift method}
Let $\g$ be a finite-dimensional Lie algebra over $\Complex$, and let $x \in \g^*$. Let $$\g_x = \{  \xi \in \g \mid \ad^*_\xi(x) = 0\}$$ be the stabilizer of $x$ w.r.t. the coadjoint representation. The set $\Sing$ of singular elements in $\g^*$ is
$$
\Sing = \{x \in \g^* \mid \dim \g_x > \ind \g\}
$$
where $$\ind \g = \min_{x \in \g^*} \dim \g_x$$
is the \textit{index} of $\g$.\par
We consider the case when $\Sing$ has codimension one. Let $\Sing_0$ be the union of all irreducible components of $\Sing$ which have maximal dimension. The set $\Sing_0$ is the zero set of a certain homogenous polynomial $ p_\g(x)$. It is easy to see that $p_\g$ is a semi-invariant of the coadjoint representation. Following Ooms and Van den Bergh \cite{Ooms}, we call it the \textit{fundamental semi-invariant} of $\g$ (see also Joseph and Shafrir \cite{JS}). \par
More precisely, the fundamental semi-invariant is defined as follows. Let $t =  \dim \g - \ind \g$. Fix a basis in $\g$, and let $c_{ij}^k$ be the structure constants in this basis. Then an element $x \in \g^*$ is singular if and only if the rank of the matrix $\mathcal A_x = c_{ij}^kx_k$ is less than $t$, i.e. if Pfaffians of all principal $t \times t$ minors of $\mathcal A_x$ vanish. Define $p_\g$ as the greatest common divisor of all these Pfaffians. Clealry, the zero set of $p_\g$ coincides with $\Sing_0$. However, $p_\g$ is not necessarily the minimal polynomial which define $\Sing_0$.

 \par
%Now define the extended Mischenko-Fomenko subalgebra. 
Now we use the fundamental semi-invariant $p_\g$ to define the extended  Mischenko-Fomenko subalgebra.
Consider the polynomial $p_\g(a+\lambda x)$ and expand it in powers of $\lambda$:
$$
p_\g(a+\lambda x) =  \sum_{i=0}^{n} p_{i}(x)\lambda^i
$$
where $n = \deg p_\g$. Define the \textit{extended  Mischenko-Fomenko subalgebra} $\widetilde \F_a$ as a subalgebra in $S(\g)$ generated by all elements of the classical Mischenko-Fomenko subalgebra $\mathcal F_a$ and the polynomials $p_1, \dots, p_{n}$.
\begin{theorem}\label{thm1}
For each regular $a \in \g^*$, the {extended  Mischenko-Fomenko subalgebra} $\widetilde \F_a$ is commutative with respect to both brackets $\{f,g\}$ and $\{f,g\}_a$.
\end{theorem}
\begin{remark}
This statement is, in fact, not new.
Firstly, it was proved by Arhangel'ski\v\i \,\,\cite{Arh} that if two semi-invariants $f,g$ commute with respect to the Lie-Poisson bracket, then their shifts, i.e. functions of the form $f(x+\lambda a)$, $g(x + \mu a)$ where $a \in \g^*$ is fixed, also commute. This statement easily implies that  the extended  Mischenko-Fomenko subalgebra $\widetilde \F_a$ is commutative with respect to the Lie-Poisson bracket. Moreover, the proof of Arhangel'ski\v\i \, can be easily modified to show that $\widetilde \F_a$  is commutative with respect to the frozen argument bracket as well (see Section \ref{pt1}). We also note that the assumption that $f$ and $g$ commute is in fact satisfied for arbitrary semi-invariants (see Ooms and Van den Bergh \cite{Ooms} and Section \ref{pt1}). \par
Secondly, Theorem \ref{thm1} can be viewed as a generalization of Proposition 7 of Bolsinov and Zhang \cite{BZh} which asserts that the functions $p_1, \dots, p_n$ are in involution with respect to both Lie-Poisson bracket and frozen argument bracket.\par
Thirdly, Theorem \ref{thm1} follows from the following general construction from the theory of compatible Poisson brackets. 
Let $\mathcal A$ and $\mathcal B$ be compatible Poisson brackets. Consider the family $\mathcal F$ generated by Casimir functions of all generic linear combinations of $\mathcal A$ and $\mathcal B$. This family is commutative with respect to both $\mathcal A$ and $\mathcal B$ (see Reiman and Semenov-Tyan-Shanskii \cite{Reiman}), however it may be incomplete or even empty. In this case, $\F$ can be extended by adding eigenvalues of the operator $\mathcal A \mathcal B^{-1}$ (which is still well-defined on a certain quotient space even if $\mathcal B$ is degenerate, see e.g. \cite{sbs}).
%Extend $\F$ by  adding eigenvalues of the operator $\mathcal A \mathcal B^{-1}$ (note that this operator is well-defined on a certain quotient space even if $\mathcal B$ is degenerate). 
Denote the extended family by $\widetilde \F$.
%$$
%\rank\left( \mathcal A(x) - \lambda_i(x) \mathcal B(x) \right)< \max_\lambda \rank \left( \mathcal A(x) - \lambda\mathcal B(x)\right).
%$$
%\end{enumerate}
%the rank of $\mathcal A(x) - \lambda_i(x) \mathcal B(x)$ is less than the rank of a generic linear combination of $\mathcal A$ and $\mathcal B$. 
\begin{statement}\label{prop1}
The family $\widetilde \F$ is commutative with respect to both brackets $\mathcal A$ and $\mathcal B$.
\end{statement}
If we apply this construction to the Lie-Poisson bracket and the frozen argument bracket, we get Theorem \ref{thm1}.\par
%We assert that these eigenvalues commute with each other and with Casimir functions of generic brackets of the pencil. 
Although we were not able to find the statement of Proposition \ref{prop1} in the literature,  we believe that it is well known to experts in the field.  See \cite{Magri, GD, Reiman, MM} where different constructions of integrable systems related to compatible Poisson brackets are discussed. Also note that the relation between the argument shift method and compatible Poisson brackets was probably first mentioned by Meshcheryakov \cite{Meshcheryakov}. 

\end{remark}
\section{Completeness criterion}
Let $\Sing_{sr} \subset \Sing_0$ be the subset of $\Sing_0$ which consists of subregular elements, i.e.
$$
\Sing_{sr} = \{x \in \Sing_0 \mid \dim \g_x = \ind \g + 2\}.
$$
It is clear that $\Sing_{sr}$ is open in $\Sing_0$, however it is not necessarily dense and may be empty. Let also $\Sing_{sm} \subset \Sing_0$ be the set of points where $\Sing_0$ is smooth. This set is open and dense in $\Sing_0$.\par
Denote the two-dimensional non-Abelian Lie algebra by $\mathfrak{b}_2$, and the $2n+1$-dimensional Heisenberg algebra by  $\mathfrak{h}_{2n+1}$.
\begin{statement}\label{class}
Let $x \in \Sing_{sr} \cap \Sing_{sm}$. Then $\g_x$ is isomorphic to one of the following Lie algebras:
\begin{enumerate}
\item $\mathfrak{b}_2 \oplus \mbox{Abelian Lie algebra of dimension } \ind \g$;
\item $\mathfrak{h}_3  \oplus \mbox{Abelian Lie algebra of dimension } \ind \g - 1$;
\item Abelian Lie algebra of dimension $\ind \g + 2$.
\end{enumerate}
\begin{proof}
Corollary 2.1 of \cite{deral} implies that the derived algebra of $\g_x$ is at most one-dimensional. It is easy to see that any Lie algebra with this property is either Abelian, or isomorphic to one of the following:
\begin{itemize}
\item $\mathfrak{b}_2 \oplus \mbox{Abelian}$; 
\item $\mathfrak{h}_{2n+1}  \oplus \mbox{Abelian}$;
\end{itemize}
Now, using that $\dim \g_x = \ind \g+2$ and the inequality $\ind \g_x \geq \ind \g$, we obtain the above list.
%Taking into account that $\dim \g_x = \ind \g+2$, it suffices to show that $n=1$ in case (b).
%By $\mathfrak{h}_n$ we denote the $n$-dimensional Heisenberg algebra.
% a direct sum of $\mathfrak{b}_2$ and an Abelian Lie algebra, or isomorphic to a direct sum of $\mathfrak{h}_n$ and an Abelian Lie algebra where $\mathfrak{h}_n$ is the $n$-dimensional Heisenberg algebra.
\end{proof}
 
\end{statement}
\begin{remark}
The inequality $\ind \g_x \geq \ind \g$ is true for any Lie algebra $\g$ and any $x \in \g^*$. It is known as the Vinberg inequality. See Panyushev \cite{Pan}.
\end{remark}
Consider the set
$$\Sing_{\mathfrak b} = \{ x \in \Sing_{sr} \cap \Sing_{sm} \mid \g_x \simeq \mathfrak b_2 \oplus \mbox{Abelian}\} \subset \Sing_{sr} \cap \Sing_{sm} \subset \Sing_0.$$
It is easy to see that $\Sing_{\mathfrak b} $ is open in $\Sing_{sr} \cap \Sing_{sm}$ and thus in $\Sing_0$. \par
%Let $a \in \g^*$. Define
%$$
% \Sing_{\mathfrak b, a} = \{x \in \Sing_{\mathfrak b}  \mid a \mbox{ is transversal to }\Sing_{\mathfrak b} \mbox{ at  }x \}.
%$$
%Then the set $ \Sing_{\mathfrak b, a}$ is also open in $\Sing_0$.
\begin{theorem}\label{thm2}
The extended  Mischenko-Fomenko subalgebra $\widetilde \F_a$ is complete if and only if the set $ \Sing_{\mathfrak b}$ is dense in $\Sing_0$.
\end{theorem}
\begin{consequence}
Assume that for each irreducible component $S_i$ of the variety $\Sing_0$ there exists at least one $x \in S_i$ such that $\g_x \simeq \mathfrak b_2 \oplus \Complex^{\ind \g}$. Then the extended  Mischenko-Fomenko subalgebra $\widetilde \F_a$ is complete. In particular, if $\Sing_0$ is irreducible, and there exists at least one $x \in \Sing_0$ such that $\g_x \simeq \mathfrak b_2 \oplus \Complex^{\ind \g}$, then $\widetilde \F_a$ is complete.
\end{consequence}\newpage
\begin{remark}
Theorem \ref{thm2} can also be formulated as follows.  For each irreducible component $S_i$ of the variety $\Sing_0$, there exists an open subset $\widetilde S_i$ such that all elements of this subset have isomorphic stabilizers. The extended  Mischenko-Fomenko subalgebra $\widetilde \F_a$ is complete if and only if $\g_x \simeq \mathfrak b_2 \oplus \Complex^{\ind \g}$ for each $i$ and each $x \in \widetilde S_i$.
\end{remark}
We also note that Theorem \ref{thm2} can be generalized to the case of arbitrary compatible Poisson brackets. %In this case, the stabilizer $\g_x$ should be replaced by the 
%The proof of Theorem \ref{thm2} will be published later.\par \smallskip
%\begin{consequence}
%Assume that $\Sing_{\mathfrak b}$ is dense in $\Sing_0$. Then the extended  Mischenko-Fomenko subalgebra $\widetilde \F_a$ is complete for almost all regular $a \in \g^*$.
%\end{consequence}
%\begin{consequence}
%Assume that $\Sing_0$ is irreducible, and that $\g_x \simeq \mathfrak b_2 \oplus \mbox{Abelian}$. Then the extended  Mischenko-Fomenko subalgebra $\widetilde \F_a$ is complete for almost all regular $a \in \g^*$.
%\end{consequence}
%According to Proposition \ref{class}, the set $\Sing_{sr} \cap \Sing_{sm}$ can be represented as a disjoint union
%$$\Sing_{sr} \cap \Sing_{sm} = \Sing_b \sqcup \Sing_h \sqcup \Sing_a$$
%where
%\begin{align*}
%\Sing_b &= \{ x \in \Sing_{sr} \cap \Sing_{sm} \mid \g_x \simeq \mathfrak b_2 \oplus \mbox{Abelian}\}, \\ \Sing_h &= \{ x \in \Sing_{sr} \cap \Sing_{sm} \mid \g_x \simeq \mathfrak h_3 \oplus \mbox{Abelian}\},\\
%\Sing_a &= \{ x \in \Sing_{sr} \cap \Sing_{sm} \mid \g_x  \mbox{ is Abelian}\}
%\end{align*}
 %Mischenko-Fomenko subalgebras $\F_a$ have the following remarkable property: they are commutative with respect to both Poisson structures $\{f,g\}$ and $\{f,g\}_a$.
 \section{Proof of Theorem 1}\label{pt1}
 Though the statement of Theorem 1 follows from from considerations of Sections \ref{shiftsSec} and \ref{LA}, we give an independent algebraic proof. As a matter of fact, we prove a stronger statement: if $\g$ is a complex Lie algebra, and $a\in \g^*$, then for any two semi-invariants $f,g \in S(\g)$ and any $\lambda, \mu \in \Complex$, 
 \begin{align*}\{f(a+\lambda x),g(a+\mu x) \}=0, \quad \{f(a+\lambda x),g(a+\mu x) \}_a = 0.\end{align*}The proof given below follows the ideas of Arhangel'ski\v\i \,\,\cite{Arh}.\par\smallskip
  Recall that $f \in S(\g)$ is called a semi-invariant of $\g$ if there exists a character $\chi_f \colon \g \to \Complex$ such that
 $$
 \{f,g\} = \chi_f(\diff g)f
 $$
 for any $g \in S(\g)$.   \begin{statement}
Assume that $\{f,g\}$ is divisible by $f$ for any $g \in S(\g)$. Then $f$ is a semi-invariant.
 \end{statement}
 \begin{proof}
 Let $x_1, \dots, x_n$ be a basis in $\g$. Then 
\begin{align}\label{three}\{f,g\} = \sum_i\{f,x_i\}\diffFX{g}{x_i}. \end{align}
  Since  $\{f,x_i\}$ is divisible by $f$ and has the same degree as $f$,  there exists $c_i \in \Complex$ such that $ \{f,x_i\} = c_i f$. Define a linear function  $\chi_f \colon \g \to \Complex$ by setting  $
  \chi_f(x_i) = c_i.
  $
  Then \eqref{three} can be rewritten as
   $$
 \{f,g\} = \chi_f(\diff g)f.
 $$
  Further,
  $$
  \{f,\{x_i,x_j\}\} =   \{\{f,x_i\},x_j\} +   \{x_i,\{f,x_j\}\} = c_i\{f,x_j\} + c_j\{x_i,f\} = c^ic^jf-c^jc^if = 0,
  $$
  so
  $$
  \chi_f(\{x_i,x_j\}) = \chi_f([x_i,x_j]) = 0,
  $$
  i.e. $\chi_f$ is a character, q.e.d.
 %For $\xi \in \g$, let $\chi_f()$
% We have $\{f,g\} = \{f,x^i\}\pdd{g}{x^i} $ 
 \end{proof}
 \begin{statement}
Any semi-invariant $f$ is a product of irreducible semi-invariants.  
 \end{statement}
 \begin{proof}
 Let $f = f_1^{k_1}  \ldots  f_m^{k_m}$ where $f_1, \dots, f_m$ are irreducible and distinct. Then 
 \begin{align}\label{one}
 \{f,g\} = \{ f_1^{k_1}  \ldots  f_n^{k_m}, g\} = \sum_{i=1}^m k_i f_1^{k_1}  \ldots  f_i^{k_i - 1}\ldots  f_m^{k_m}\{f_i,g\}.
\end{align}
 On the other hand,
\begin{align}\label{two}
 \{f,g\} = \chi_f(\diff g)f =  \chi_f(\diff g) f_1^{k_1}  \ldots  f_m^{k_m}.
\end{align}
 Comparing \eqref{one} and \eqref{two}, we conclude that $\{f_i,g\}$ is divisible by $f_i$, hence $f_i$ is a semi-invariant.
  \end{proof}
  \begin{statement}
  Let $f,g$ be two semi-invariants. Then $\{f,g\} = 0$.
  \end{statement}
  \begin{proof}
  Clearly, it suffices to prove the proposition for irreducible distinct  $f$ and $g$. Let $h = \{f,g\}$. Then $h$ is divisible by both $f$ and $g$, and hence by $fg$. On the other hand, 
  $$
  \deg h \leq \deg f + \deg  g- 1,
  $$
  therefore $h=0$.
    \end{proof}
    \begin{statement}\label{shift1}
  Let $f,g$ be two semi-invariants. Then $\{f,g(a+\lambda x)\} = 0$.
  \end{statement}
  \begin{proof}
 Since $\{f,g\} = 0$, we have $\chi_f(\diff g(x)) = 0$ for any $x$. Let $h(x) = g(a+\lambda x)$. Then 
 $$
 \diff h(x) = \lambda \diff f(a+\lambda x),
 $$
 so $\chi_f(\diff h(x)) = 0$, and $\{f,h\} = \chi_f(\diff h)f = 0$.
 
  \end{proof}
  \newpage
      \begin{statement}
  Let $f,g$ be two semi-invariants. Then $\{f,g(a+\lambda x)\}_a = 0$.
  \end{statement}
  \begin{proof}
   We shall use the following explicit formulas for the Lie-Poisson and frozen argument bracket:
 $$
 \{f,g\}(x) = \langle x,[\diff f(x), \diff g(x)]\rangle, \quad  \{f,g\}_a(x) = \langle a, [\diff f(x), \diff g(x)]\rangle.
 $$
Let $h(x) = g(a+\lambda x)$. Then 
$$
\{f,h\}_a(x) = \langle a, [\diff f(x), \diff h(x)]\rangle = \lambda \langle a,[\diff f(x), \diff g(a+\lambda x)]\rangle. 
$$
On the other hand, 
$$
\lambda \langle x, [\diff f(x), \diff g(a+\lambda x)] \rangle=   \{f,h\}(x) = 0,
$$
so
 $$
\{f,h\}_a(x) =  \lambda \langle a+\lambda x,[\diff f(x), \diff g(a+\lambda x)]\rangle,
 $$
 and
  $$
\{f,h\}_a\left(\frac{x-a}{\lambda} \right) =  \lambda \langle x ,[\diff f\left(\frac{x-a}{\lambda} \right), \diff g(x)] \rangle = \lambda^2\{f\left(\frac{x-a}{\lambda} \right),g\}(x).
 $$
 The latter Poisson bracket vanishes by Proposition \ref{shift1}, so
 $$
 \{f,h\}_a\left(\frac{x-a}{\lambda} \right) = 0
 $$
 for any $x$, and hence $\{f,h\}_a = 0$.
  \end{proof}
  \begin{proof}[Proof of Theorem 1]
  Let $f,g$ be two semi-invariants. We need to prove that 
  $$\{f(a+\lambda x),g(a+\mu x)\} = \{f(a+\lambda x),g(a+\mu x)\}_a= 0.$$
  Let $h(x) = f(a+\lambda x), k(x) = g(a+\mu x)$. Then
  $$
  \{h,k\}(x) = \langle x,[\diff h(x), \diff k(x)]\rangle= \lambda \mu \langle x, [\diff f(a+\lambda x), \diff g(a+\mu x)]\rangle, 
  $$
  so
\begin{align*}
  \{h,k\}\left(\frac{x-a}{\lambda} \right)  &= \lambda \mu \langle \frac{x-a}{\lambda}, [\diff f(x), \diff g\left(a + \frac{\mu}{\lambda} (x-a)\right)]\rangle = \\ &= \lambda\{f, g\left(a + \frac{\mu}{\lambda} (x-a)\right)\}(x) - \lambda \{f, g\left(a + \frac{\mu}{\lambda} (x-a)\right)\}_a(x) = 0,
\end{align*}
and hence $\{h,k\} = 0$.
Analogously,
\begin{align*}
  \{h,k\}_a\left(\frac{x-a}{\lambda}\right)  = \lambda \mu \langle a ,[\diff f(x), \diff g\left(a + \frac{\mu}{\lambda} (x-a)\right)]\rangle = \lambda^2 \{f, g\left(a + \frac{\mu}{\lambda} (x-a)\right)\}_a(x) = 0,
\end{align*}
and hence $\{h,k\}_a = 0$.
  \end{proof}
  
  \section{Shifts of the fundamental semi-invariant}\label{shiftsSec}
  Recall that the shifts $p_1, \dots, p_n$ of the fundamental semi-invariant are defined by the formula:
  
$$
p_\g(a+\lambda x) =  \sum_{i=0}^{n} p_{i}(x)\lambda^i
$$
where $n = \deg p_\g$.
  Consider the factorization of $p_{\g}$ into irreducible factors:
  $$
  p_{\g} = p_1^{k_1}\dots p_m^{k_m},
  $$
  and let $d =\sum \deg p_i$.
  Let $x \in \g^*$, and consider the equation
  $$
  p_{\g}(x - \lambda a) = 0.
  $$
  Obviously, this equation has at most $d$ distinct roots.
  Let us say that an element $x \in \g^*$ is nice if it has the following properties:
  \begin{enumerate}
  \item the equation $p_{\g}(x - \lambda a) = 0$ has exactly $d$ distinct roots $\lambda_1, \dots, \lambda_d$;
    \item $\lambda_1, \dots, \lambda_d$ are locally analytic functions of $x$;
    \item for each $i$, the dimension of the stabilizer of ${x - \lambda_i(x)a}$ is locally constant;
  \item the line $x-\lambda a$ does not intersect the set $\Sing \setminus \Sing_0$.
  \end{enumerate}
  
  It is clear that the set $\mathcal N$ of nice elements is Zariski dense in $\g^*$.
  \newpage
  \begin{statement}\label{spandiff}
  Let $x \in \mathcal N$. Then the space spanned by the differentials of $\lambda_1, \dots, \lambda_d$ coincides with the space spanned by the differentials of $p_1, \dots p_n$.
  \end{statement}
 \begin{proof}
 Let $s_i$ be the multiplicity of $\lambda_i$. Then
  $$
    p_{\g}(x - \lambda a) =p_\g(a)\prod_{i=1}^d (\lambda_i(x) - \lambda)^{s_i},
  $$
  so
  $$
  p_{\g}(a+\lambda x) = \lambda^n p_{\g}(x + \frac{1}{\lambda}a) = p_\g(a)\prod_{i=1}^d (1 + \lambda_i(x)\lambda)^{s_i},
  $$
therefore the functions $p_1, \dots, p_n$ are, up to a constant factor, elementary symmetric polynomials of the functions $\lambda_1, \dots, \lambda_d$ taken with multiplicities, which easily implies the statement.
  \end{proof}
  Let $\g_i(x) = \g_{x - \lambda_i(x)a}$ be the stabilizer of $x - \lambda_i(x)a$.   The two following statements relate the differentials $\lambda_1, \dots, \lambda_d$ to the structure of stabilizers $\g_1(x), \dots, \g_d(x)$.
    \begin{statement}\label{preFundId}
  Let $x_0 \in \mathcal N$. Then $\diff \lambda_i(x_0) \in \g_{i}(x_0)$.
  \end{statement}
  \begin{proof}
Consider the coadjoint orbit $\mathcal O$ passing through $y_0 = x_0 - \lambda_i(x_0)a$, and let $\xi \in \T_{y_0}\mathcal O $ be a vector tangent to the orbit at $y_0$. Let also $y(t)$ be a curve such that $y(t) \in \mathcal O, y(0) = y_0,$ and $\dot y(0) = \xi$. Let also $x(t) =  y(t) + \lambda_i(x_0)a$. Obviously, $y(t) \in \Sing$, and $\lambda_i(x(t)) = \lambda_i(x_0)$. Differentiating this formula with respect to $t$ at $t = 0$, we obtain $$\langle \dot x(0), \diff \lambda_i(x_0) \rangle = 0,$$ and since $\dot x(0) = \dot y(0) = \xi$, we have
\begin{align}\label{ann}
 \langle \xi, \diff \lambda_i(x_0) \rangle = 0.
\end{align}
 Since \eqref{ann} is true for any $\xi \in \T_{y_0}\mathcal O $,  we have \begin{align*}\langle  \T_{y_0}\mathcal O , \diff \lambda_i(x_0) \rangle = 0,\end{align*} which implies that $\diff \lambda_i(x_0) \in \g_{y_0} = \g_i(x_0)$.
  \end{proof}
  The following simple formula is of fundamental significance for the present paper.
  \begin{statement}\label{fundIdP}
  Let $x_0 \in \mathcal N$, and let $\xi, \eta \in \g_{i}(x_0)$. Then
\begin{align}\label{fundId}
  [\xi, \eta] = \langle a,  [\xi, \eta] \rangle \, \diff \lambda_i(x_0).
\end{align}
  \end{statement}
  \begin{proof}
 Choose a neighborhood $U(x_0) \ni x_0$ such that the dimension of $\g_{i}(x)$ is constant in $U(x_0)$. Then it is possible to define smooth mappings $\xi, \eta \colon U(x_0) \to \g$ such that
  $\xi(x_0) = \xi, \eta(x_0) = \eta$, and $\xi(x), \eta(x) \in \g_{i}(x)$
  for each $x \in U(x_0)$. Differentiating the identity
  $$
  \langle x - \lambda_i(x)a, [\xi(x), \eta(x)]\rangle = 0.
  $$
at $x = x_0$, we obtain \eqref{fundId}.
  \end{proof}
  %The functions $\lambda_1, \dots, \lambda_d$
  \section{Linear algebra related to a pair of skew-symmetric forms}\label{LA}
  Let $P_0, P_\infty$ be two skew-symmetric forms on a vector space $V$, and let
  $$
  P_\lambda = P_0 - \lambda P_\infty.
  $$
  Let also
\begin{align*}
 r =  \min_{\lambda \in \overline \Complex} \corank P_\lambda, \quad
  \Lambda = \{ \lambda \in \overline \Complex \mid \corank P_\lambda> r \}, \quad
  L = \sum_{\lambda \in \overline \Complex \setminus \Lambda} \Ker P_\lambda.
\end{align*}
%\begin{statement}\label{LA1}
%Let $\alpha, \beta \in \overline \Complex, \alpha \neq \beta$. Then $P_\lambda(\Ker P_\alpha, \Ker P_\beta) = 0$ for each $\lambda \in \overline \Complex$.
%\end{statement}
  \begin{statement}\label{LA2}
The space $L$ has the following properties:
\begin{enumerate}
       \item it is isotropic with respect to any form $P_\lambda$;
       \item  the skew-orthogonal complement to $L$ given by $L^\bot = \{ \xi  \in V \mid P_\lambda(\xi,L) = 0\}$ does not depend on the choice of $\lambda \in \overline\Complex$;
       \item if $\lambda \notin \Lambda$, then $P_\lambda$ is non-degenerate on $L^\bot /L$;
       \item if $\lambda \in \Lambda$, then $\dim\, (\Ker P_\lambda \cap L) = r$;
        \item if $\lambda \in \Lambda$, and $\alpha \in \overline \Complex$, then	$
	 \Ker \left( P_\alpha \mid_{\Ker P_\lambda}  \right) \supset \Ker P_\lambda \cap L.
	$

%       \item Let $k \geq \dim L$. Then for any distinct $\alpha_{1}, \dots, \alpha_{k} \in \overline \R \setminus \Lambda$ the following equality holds
%	$$
%		\sum\limits_{i=1}^{k} \Ker P_{\alpha_{i}} = L.
%	$$
	%\item $\dim (\Ker P_{\lambda} \cap L) = \corank \Pi(x)$ for all $\lambda\in \RP$. \item Similarly, $\dim_\Complex (\Ker P_{\lambda} \cap L \otimes \Complex) = \corank \Pi(x)$ for all $\lambda\in \Complex \setminus \R$.
	\end{enumerate}
	\newpage
\end{statement}
Assume that $\infty \notin \Lambda$, and define the \textit{recursion} operator $R \colon L^\bot / L \to L^\bot L$ by the formula $$R = P_\infty^{-1}P_0.$$
\begin{statement} \label{LA3}The operator $R$ has the following properties:
	\begin{enumerate}
	\item the spectrum of $R$ coincides with the set $\Lambda$; the multiplicity of each eigenvalue is at least two;
	\item the $\lambda$-eigenspace of $R$ coincides with the space
	$
		(\Ker P_{\lambda}) / (\Ker P_{\lambda} \cap L)
	$;
	\item the eigenspaces of $R$ are pairwise orthogonal with respect to $P_\lambda$ for each $\lambda$;
	\item the operator $R$ is diagonalizable if and only if for each $\lambda \in \Lambda$, the following identity holds
	$$
	\dim \Ker \left( P_\infty \mid_{\Ker P_\lambda}  \right) = r.
	$$
	\end{enumerate}
\end{statement}
For the proof of Propositions \ref{LA2}, \ref{LA3}, see \cite{sbs}. They can also be easily deduced from the Jordan-Kronecker theorem \cite{Gurevich,Thompson, GZ2, Kozlovj}.
\begin{statement}\label{LA4}
Let $\Lambda = \{\lambda_1, \dots, \lambda_k\}$, and assume that $\xi_i \in \Ker P_{\lambda_i}$. Let
$$
U = L + \langle \xi_1, \dots, \xi_k \rangle.
$$
Then
\begin{enumerate}
\item $U$ is isotropic with respect to $P_\lambda$ for any $\lambda \in \overline \Complex$;
\item if $\lambda \notin \Lambda$, then $U$ is maximal isotropic with respect to $P_\lambda$ if and only if all eigenvalues of $R$ have multiplicity two, and $\xi_i \notin L$ for each $i$.
\end{enumerate}
\end{statement}
The proof easily follows from Propositions \ref{LA2} and \ref{LA3}.
    %Then $A,B$ 
  %  Then the Poisson brackets $\{\,,\}$ and $\{\,,\}_a$ define skew-symmetric forms $\mathcal A_x$ and $\mathcal A_a$ on the cotangent space $\T^*_x \g^* \simeq \g$. These forms are given by
%  $$
%  A_x(\xi, \eta) = x([\xi,\eta]), \quad   A_a(\xi, \eta) = a([\xi,\eta]).
%  $$
\section{Proof of Theorem 2}
Assume that $ \Sing_{\mathfrak b}$ is dense in $\Sing_0$, and prove that the extended Mischenko-Fomenko subalgebra $\widetilde{\mathcal F_a}$ is complete. Let us take $x \in \mathcal N$ such that $x - \lambda_i(x)a \in  \Sing_{\mathfrak b}$ for each $i$, and prove that the space
$$
\diff \widetilde{\mathcal F}(x) = \{ \diff f(x) \mid f \in  \widetilde{\mathcal F}(x)\}
$$
has dimension
 $$
b(\g) =  \frac{1}{2}(\dim \g + \ind \g),
$$
which immediately implies the completeness of $\widetilde{\mathcal F_a}$.\par
Consider skew-symmetric forms $P_0 = \mathcal A_x$ and $P_\infty = \mathcal A_a$ on the cotangent space $\T^*_x \g^* \simeq \g$ which are given by
 $$
A_x(\xi, \eta) = \langle x, [\xi,\eta] \rangle, \quad   A_a(\xi, \eta) = \langle a,[\xi,\eta]\rangle.
$$
Let us apply the results of Section \ref{LA} to these two forms. We note that
\begin{align*}
r = \mathrm{ind}\, \g, \quad \Ker P_\lambda = \g_{x - \lambda a}, \quad \Lambda = \{\lambda_1(x), \dots, \lambda_d(x) \}.
\end{align*}
The following lemma is due to Bolsinov \cite{Bolsinov}.
\begin{lemma}
Let $\mathcal F$ be the classical Mischenko-Fomenko subalgebra, and let
$$
\diff {\mathcal F}(x) = \{ \diff f(x) \mid f \in  {\mathcal F}(x)\}.
$$
Then
$$
\diff {\mathcal F}(x) = L.
$$
\end{lemma}
As follows from Proposition \ref{spandiff}, $$\diff \widetilde{\mathcal F}(x) = \diff {\mathcal F}(x) + \langle \diff \lambda_1(x), \dots, \diff \lambda_d(x) \rangle.$$ By Proposition \ref{preFundId}, we have $$\diff \lambda_i(x) \in \g_{i}(x) = \Ker P_{\lambda_i},$$ so we can use Proposition \ref{LA4} to show that $\diff \widetilde{\mathcal F}(x) $ is maximal isotropic with respect to $A_a$ and hence is of dimension $b(\g)$. In order to do this, we need to show that the eigenvalues of the recursion operator $R$  have multiplicity two, and that $\diff \lambda_i(x) \notin L$.\par
Since $x - \lambda_i(x)a \in  \Sing_{\mathfrak b}$, we have
$
\dim \g_i =r + 2,
$
and by item 4 of Proposition \ref{LA2}, $\dim\left( \g_i \cap L \right) = r$, so $$\dim \left( \g_i / ( \g_i \cap L) \right)= 2,$$
and all eigenspaces of $R$ are two-dimensional (see Proposition \ref{LA3}, item 2). Therefore, to prove that all eigenvalues of $R$ have multiplicity two, we need to show that $R$ has no Jordan blocks.
By {item 4} \\ of Proposition \ref{LA3}, it suffices to prove that
$
\dim \Ker (A_a \mid \g_i) = r. 
$
We have
$
\dim \g_i =r + 2,
$
and
$
\dim \g_i \cap L = r.
$
By item 5 of Proposition \ref{LA2}, $ \Ker (A_a \mid \g_i) \supset \g_i \cap L$, so $\dim \Ker (A_a \mid \g_i)$ can be either $r$ or $r+2$. Assume that it is $r+2$. Then $ \Ker (A_a \mid \g_i) = \g_i$, and $A_a \mid \g_i = 0$. By Proposition \ref{fundIdP}, this implies that $\g_i$ is Abelian, which is not the case.\par
Now, let us prove that $\diff \lambda_i(x) \notin L$. Since $ \Ker (A_a \mid \g_i) \supset \g_i \cap L$, Proposition \ref{fundIdP} implies that $ \g_i \cap L$ lies in the center $\mathcal Z({\g_i})$. So, if $\diff \lambda_i(x) \in L$, then
$$
\diff \lambda_i(x) \in \mathcal Z({\g_i}).
$$
On the other hand, since $\g_i$ is not Abelian, Proposition \ref{fundIdP} implies that 
$$
\diff \lambda_i(x) \in [\g_i, \g_i], \quad \diff \lambda_i(x) \neq 0,
$$ 
where $ [\g_i, \g_i]$ is the derived subalgebra of $\g_i$. But since $\g_i \simeq \mathfrak b_2 \oplus \mbox{Abelian}$, we have $[\g_i, \g_i] \cap  \mathcal Z({\g_i}) = 0$, so $\diff  \lambda_i(x) \notin L$, which completes the proof of the if-part of the theorem. The proof of the only-if-part is analogous.

\section{Acknowledgements}
The author is grateful to Alexey Bolsinov for useful comments. This work was supported by the Dynasty Foundation Scholarship.
\bibliographystyle{unsrt} 
\bibliography{gasm} 
\end{document}